\documentclass[11pt]{article}
\usepackage{amsmath,amsthm,verbatim,amsfonts,amscd, graphicx}
\usepackage{hyperref}
\usepackage{parskip}
\usepackage{xparse}

\usepackage[utf8]{inputenc}    
\usepackage[T1]{fontenc} 
\usepackage{tikz-cd}
\usepackage{mathtools} 
\usepackage{mathrsfs}
\usepackage{enumerate} 
\usepackage{braket}
\usepackage{bigints}
\usepackage{amssymb}
\let\amslrcorner\lrcorner
\usepackage{MnSymbol}
\let\lrcorner\amslrcorner

\theoremstyle{plain}
\newtheorem{theorem}{Theorem}[section]
\newtheorem{lemma}{Lemma}[section]

\newtheorem{proposition}{Proposition}[section]

\newtheorem*{claim*}{Claim}
\newtheorem*{lemma*}{Lemma}
\newtheorem*{theorem*}{Theorem}
\newtheorem{maintheorem}{Theorem}

\theoremstyle{definition}
\newtheorem{definition}{Definition}[section]

\theoremstyle{remark}
\newtheorem{remark}{Remark}

\DeclareMathOperator{\Hessian}{Hess}

\DeclareMathOperator{\Domain}{Dom}

\DeclareMathOperator{\Range}{Ran}

\renewcommand{\Im}{\operatorname{Im}}
\renewcommand{\Re}{\operatorname{Re}}
\renewcommand{\i}{\operatorname{\sqrt{-1}}}

\newcommand{\PSH}{\mathrm{PSH}}

\makeatletter
\newcommand{\extp}{\@ifnextchar^\@extp{\@extp^{\,}}}
\def\@extp^#1{\mathop{\bigwedge\nolimits^{\!#1}}}
\makeatother

\allowdisplaybreaks
\everymath{\displaystyle}

\NewDocumentCommand{\inn}{mo}{%
	\langle #1\rangle
	\IfValueT{#2}{^{}_{\mspace{-3mu}#2}}%
}

\NewDocumentCommand{\dinn}{mo}{%
	\llangle #1\rrangle
	\IfValueT{#2}{^{}_{\mspace{-3mu}#2}}%
}

\NewDocumentCommand{\xinn}{>{\SplitArgument{1}{,}}mo}{%
	\doinnerproduct#1
	\IfValueT{#2}{^{}_{\mspace{-3mu}#2}}%
}
\NewDocumentCommand{\doinnerproduct}{mm}{%
	\langle #1\mid #2\rangle 
}

\makeatletter
\newcommand*\bigcdot{\mathpalette\bigcdot@{1}}
\newcommand*\bigcdot@[2]{\mathbin{\vcenter{\hbox{\scalebox{#2}{$\m@th#1\bullet$}}}}}
\makeatother

\usepackage[
backend=bibtex,  
natbib=true,
url=false, 
]{biblatex}
\IfFileExists{mybibliography.bib}
{\addbibresource{mybibliography.bib}}
{}
\IfFileExists{../bib/mybibliography.bib}
{\addbibresource{../bib/mybibliography.bib}}
{}

\begin{document}
	
	\title{On the H\"{o}rmander's estimate}
	
	\author{Bingyuan Liu}

	\date{\today}

	\maketitle	
	
	\begin{abstract}
	The motivation of the note is to obtain a H\"{o}rmander-type $L^2$ estimate for $\bar\partial$ equation. The feature of the new estimate is that the constant in the estimate is independent of the weight function. Moreover, our estimate can be used for non-plurisubharmonic weight function.
	\end{abstract}
	
	\section{Introduction}\label{introduction}
 
	The $L^2$-method of functional analysis has been well-applied to algebraic geometry settings for decades. It can tell if the solution of a Cauchy--Riemann equation exists. It also provides an estimate of the solution in terms of the given data. With the solution of the Cauchy--Riemann equation, one can easily construct a holomorphic function. Consequently, cohomology can be studied with the help of the $L^2$-method.

    Depending on the ambient spaces, there are two common approaches to studying the $\bar\partial$ equation: either treat the space as a manifold with a complete metric or treat it as a bounded domain in Euclidean spaces (in which case the metric on the domain will be induced from the Euclidean metric and is incomplete). One can examine the $L^2$ estimate using either of the two approaches, even for the same domain. The $L^2$ space will be smaller with complete metric than with incomplete metric. After all, the $L^2$ integrable functions/forms for the complete metric need to ``vanish'' when approaching boundary while the case of incomplete metrics needs not. As a result, the case of incomplete metrics still has yet to be well-understood. The $\bar\partial$-Neumann problem, which is related to the $\bar\partial$-problem, shares the same lack of clarity. Indeed, the necessary and sufficient conditions of the global regularity of the $\bar\partial$-Neumann problem remain unknown. See Boas--Straube \cite{BS90}, \cite{BS91}, \cite{BS91b}, \cite{BS93}, Harrington \cite{Ha11}, Harrington--Liu \cite{HL19}, Kohn \cite{Ko63}, \cite{Ko64}, Pinton--Zampiere \cite{PZ14}, and  Liu--Straube \cite{LS22} for sufficient conditions. See Huang--Li \cite{HL16} for mixed boundary conditions as well. The $\bar\partial$-Neumann problem is known to be closely related to the theory of the Bergman kernel. See Hsiao--Savale \cite{HS22}.

    The $L^2$ estimates of the $\bar\partial$ operator has wide applications in algebraic geometry. This type of estimate was initially introduced by Andreotti--Vesentini in their works \cite{AV61} and \cite{AV65}. It has been subsequently proven using different methods, by H\"{o}rmander \cite{Ho65}. More recently, Donnelly--Fefferman, Ohsawa--Takegoshi, as well as Berndtsson--Charpentier, have derived alternative forms of this estimate.

\begin{theorem}[Donnelly--Fefferman \cite{DF83}, see also Chen \cite{Ch22}]\label{1.1}
    Let $\Omega$ be a bounded pseudoconvex domain in $\mathbb{C}^n$ and $\phi\in\PSH(\Omega)$. Assume that $\phi$ is strictly plurisubharmonic so that for \[r\i\partial\bar\partial\phi\geq \i\partial\phi\wedge\bar\partial\phi\] for some $r>0$. Then for the equation $\bar\partial u=v$, where $\bar\partial v=0$ and $v\in L^2(\Omega, \phi)$, there exists a $u\in L^2(\Omega,\phi)$ so that \[\int_\Omega|u|^2e^{-\phi}\,dV\leq C_0r\int_\Omega|v|^2_{\i\partial\bar\partial\phi}e^{-\phi}\,dV.\] The $C_0$ is a constant.
\end{theorem}

\begin{theorem}[Berndtsson--Charpentier \cite{BC00},  see also Chen \cite{Ch22} and B\l ocki \cite{Bl13}]\label{1.2}
Let $\Omega$ be a bounded pseudoconvex domain in $\mathbb{C}^n$ and $\phi\in\PSH(\Omega)$. Assume that $\psi$ is strictly plurisubharmonic so that for \[r\i\partial\bar\partial\psi\geq \i\partial\psi\wedge\bar\partial\psi\] for some $r>0$. Then for the equation $\bar\partial u=v$, where $\bar\partial v=0$ and $v\in L^2(\Omega, \phi-\psi)$, there exists a $u\in L^2(\Omega,\phi-\psi)$ so that \[\int_\Omega|u|^2e^{\psi-\phi}\,dV\leq \frac{1}{(1-\sqrt{r})^2}\int_\Omega|v|^2_{\i\partial\bar\partial\phi}e^{\psi-\phi}\,dV.\] 
    
\end{theorem}

It is well-known that Ohsawa--Takegoshi present an alternative $L^2$ estimate and an extensively studied extension theory. For more details, interested readers can refer to Ohsawa--Takegoshi \cite{OT87}, Straube \cite{St10}, Demailly \cite{De12}, Chen \cite{Ch11}, B\l ocki \cite{Bl13} and Guan--Zhou \cite{GZ15}.

In this article, we prove the following:

    \begin{maintheorem}\label{mainthm}
        Let $\Omega$ be a bounded pseudoconvex domain in $\mathbb{C}^n$ and $E:=\Omega\times\mathbb{C}\rightarrow X$ be a trivial line bundle with fiber metric $e^{-\phi}$ for $\phi\in C^2(\overline{\Omega})$. Assume that there exists $\eta\in C^2(\overline{\Omega})$, $q\geq1$ and $t_2,t_3\in\mathbb{R}$ so that \[\Xi_{t_2,t_3,\eta}\geq 0.\] Let $f\in L^2_{(0,q)}(\Omega)\cap\Domain(\bar\partial)$ so that $\bar\partial f=0$, there exists $u\in\Domain(\bar\partial)$ so that \[\|u\|^2_\phi\leq (t_2+t_3)\dinn{\Xi^{-1}_{t_2,t_3,\eta}f,f}[\phi].\]
    \end{maintheorem}

Here we define the operator $\Xi$ on $(0,1)$-forms by the following. For the general definition of $(0,q)$-forms, see Section \ref{thelastsection}.

\begin{definition}
    Let $f=\sum f_idz_i\in L^2_{(0,1)}(\Omega)$. Then  \[\Xi f=(\i\frac{\partial^2(\eta+\phi)}{\partial z_i \bar z_j}-\frac{1}{4}\i\frac{\partial\eta}{\partial z_i}\frac{\partial\eta}{\partial \bar z_j})f_i\bar f_j -\frac{1}{4} |\bar\partial\eta|^2|f|^2.\]
\end{definition}
We want to remark that our formula is different from the one from Theorem \ref{1.1} and Theorem \ref{1.2}. Both Theorem \ref{1.1} and Theorem \ref{1.2} require $\phi,\psi\in\PSH(\Omega)$ but our theorem does not need this condition. Indeed, our weight function $\phi$ does not need to be plurisubharmonic and for a fixed non-plurisubharmonic weight $\phi$, we can get an estimate by varying the function $\eta$. We believe the candidate pool of $\eta$ depends on the geometric property of the domain $\Omega$. We also want to remark that the theorem of Ohsawa--Takegoshi in \cite{OT87} and Theorem \ref{1.2} can make non-plurisubharmonic weight as well. The formulation and proof of our estimate are different from theirs. In this note, we only consider trivial line bundles.

The orgnization of the note is as follows: after introduction and preliminary in Section \ref{introduction} and \ref{pre}, we introduce our new technique in Section \ref{sectionof2eq} and eventually, reformulate it to the form of Theorem \ref{mainthm} in Section \ref{thelastsection}.

	\section{Preliminary}\label{pre}
 In this section, we are going to recall the $L^2$-method of $\bar\partial$ equation in two settings: in the complete open K\"{a}hler manifold and in the bounded domain (with boundary) in $\mathbb{C}^n$ with the induced Euclidean metric.
 
 Let $I=\left(i_{1}, \ldots, i_{p}\right)$ is a multi-index with integer components, $i_{1}<\ldots<i_{p}$ and $d x_{I}:=$ $d x_{i_{1}} \wedge \ldots \wedge d x_{i_{p}} .$ The notation $|I|$ stands for the number of components of $I$. For example, $|(i_1,i_2,\cdots,i_p)|=p$.

Let $X$ be a K\"{a}hler manifold with a metric, in a local coordinate, \[g=\sum_{i,j}g_{ij}dz_i\otimes d\bar{z}_j\in\Gamma(\extp T^{*(1,0)}X\otimes \extp T^{*(0,1)}X).\] From now on, we denote the set of smooth sections $\Gamma(\extp^p T^{*(1,0)}X\wedge \extp^q T^{*(0,1)}X)$ by $\mathscr{C}^\infty_{p,q}(X,\mathbb{C})$. The associated K\"{a}hler form is \[
	\omega=-\Im g=\frac{\sum\bar{g}_{ji}d\bar{z}_j\otimes dz_i-\sum g_{ij}dz_i\otimes d\bar{z}_j}{2\i}=\frac{\sum g_{ij}(d\bar{z}_j\otimes dz_i-dz_i\otimes d\bar{z}_j)}{2\i}=\frac{\i}{2}\sum g_{ij}dz_i\wedge d\bar{z}_j.
\]

The Lefschetz operator $L$ associated with $\omega$ is defined to be $Lu=\omega\wedge u$ for $u\in\mathscr{C}^\infty_{p,q}(X,\mathbb{C})$. The Lefschetz operator is a pointwise action. It can be defined even on non-smooth sections.

The volume form $\,dV=\,dV_\omega$ associated with $\omega$ is $\,dV=\frac{\omega^n}{n!}$. For each $\omega$, we may define the Hodge-$*$ operator $*=*_\omega: \mathscr{C}^\infty_{p,q}(X,\mathbb{C})\rightarrow\mathscr{C}^\infty_{n-q,n-p}(X,\mathbb{C})$ by (Page 33 of Huybrechts \cite{Hu05}, Page 300 of Demailly \cite{De12}) \[u\wedge*\bar{v}=\inn{u,v}[\mathbb{C}]\,dV\] for any $u,v\in\mathscr{C}^\infty_{p,q}(X,\mathbb{C})$, The Hodge-$*$ operator is a pointwise action and enjoys the property \[**=(-1)^{(p+q)(2n-p-q)}\mathrm{id}\] on $\mathscr{C}^\infty_{p,q}(X,\mathbb{C})$.

We may define the formal-adjoint of $L$, $\bar{\partial}$ and $\partial$. We define $\Lambda:\mathscr{C}^\infty_{p,q}(X,\mathbb{C})\rightarrow\mathscr{C}^\infty_{p-1,q-1}(X,\mathbb{C})$ as follows: \[\llangle Lu, v\rrangle=\llangle u, \Lambda v\rrangle.\] It is clear that $\Lambda=L^*=*^{-1}L*$ is the adjoint of $L$.

For any complex manifold $X$ (even without a metric), we can define the exterior differential $d=\partial+\bar{\partial}$ on $\mathscr{C}^\infty_{p,q}(X,\mathbb{C})$. Clearly $d$, $\partial$ and $\bar{\partial}$ are independent of the K\"{a}hler metric.

The formal adjoint of $\partial$ is $\partial^*=-*\bar{\partial}*$ and this can be seen from the following calculation: for $u\in \mathscr{C}^\infty_{p-1,q}(X,\mathbb{C})$ and $v\in \mathscr{C}^\infty_{p,q;c}(X,\mathbb{C})$,
\[\int_X\inn{\partial u, v}_\mathbb{C}\,dV=\int_X\partial u\wedge *\bar{v}=-(-1)^{\mathrm{deg} u}\int_X u\wedge \partial*\bar{v}=-(-1)^{\mathrm{deg} u}\int_X \langle u, *^{-1}\bar{\partial}*v\rangle\,dV=-\int_X \inn{u, *\bar{\partial}*v}_\mathbb{C}\,dV.\]

For the same reason, $\bar{\partial}^*=-*\partial*$. It is clear that $L$, $\Lambda$, $\bar{\partial}^*$, $\,dV$ and $\partial^*$ depend on the K\"{a}hler metric.

On $X$, a differential $(p,q)$-form can be written locally, after choosing a local coordinate system $(z_i)$,
\[u=\sum_{|I|=p, |J|=q} u_{IJ} \,dz_I\wedge\,d\bar{z}_J,\] where $I=(i_1,\cdots, i_p)$ and $J=(j_1,\cdots, j_q)$ are multi-indices with integer components: the $dz_I=\,dz_{i_1}\wedge\cdots\wedge \,dz_{i_p}$ and $dz_J=\,dz_{j_1}\wedge\cdots\wedge \,dz_{j_q}$.

We want to list the following basic facts:
\begin{enumerate}
    \item The $\partial$ and $\bar\partial$ of $X$ are independent from either the metric of $X$ or the fiber metric on $E$. Indeed, they can be defined without any metrics.
    \item The $\partial_E$ and the Chern connection $\bar\partial$ on $E$ are dependent on the fiber metric on $E$. However, the $\partial_E$ and the Chern connection $\bar\partial$ on $E$ can be defined without metrics on $X$.
    \item The $\partial_E^*$ and the $\bar\partial^*$ on $E$ depend on the metric of $X$. But $\partial^*_E$ is independent from the fiber metric of $E$ in our settings (see Section \ref{sectionof2eq}).
    \item The Bochner--Kodaira--Nakano identity is a local identity and doesn't need the completeness of the metric on $X$.
\end{enumerate}

\subsection{The K\"{a}hler identity}

For a K\"{a}hler metric, at $x_0\in X$, we can find a geodesic coordinate system $(z_i)_{i=1}^n$ so that \[\omega=\i\sum_{i,j}\omega_{ij}dz_i\wedge d\bar{z}_j,\] where $\omega_{ij}=\delta_{ij}+O(|z|^2)$. This turns that, \[dV=(1+O(|z|^2))\,dV_0,\] where $dV_0$ is the volume form with the metric of $\delta_{ij}$.

In this geodesic coordinate system, \[\bar{\partial}u=\sum_{I,J,k}\frac{\partial u_{IJ}}{\partial\bar{z}_k}d\bar{z}_k\wedge dz_I\wedge d\bar{z}_J.\] Letting $u,v$ be $(p, q)$ and $(p, q+1)$-forms with compact support in a neighborhood $U(x_0)$ of $x_0$, we have
\[\llangle\bar{\partial}u, v\rrangle=\int_{U(x_0)}\sum_{I,J,k}\frac{\partial u_{IJ}}{\partial\bar{z}_k}\bar{v}_{I, kJ}+\sum_{I,J,k,K,L}a_{IJKLk}\frac{\partial u_{IJ}}{\partial\bar{z}_k}\bar{v}_{KL}\,dV_0,\] where $a_{IJKLk}(z)=O(|z|^2)$. It is clear that \[\bar{\partial}^*v=-\frac{\partial v_{IJ}}{\partial z_k}\frac{\partial}{\partial\bar{z}_k} \lrcorner (dz_I\wedge d\bar{z}_J)-a_{IJKLk}\frac{\partial v_{IJ}}{\partial z_k}dz_K\wedge d\bar{z}_L+b_{IJKLk}v_{IJ}dz_K\wedge d\bar{z}_L,\] with $b_{IJKLk}=O(|z|)$. In other words, \[\bar{\partial}^*v=-\frac{\partial v_{IJ}}{\partial z_k}\frac{\partial}{\partial\bar{z}_k} \lrcorner (dz_I\wedge d\bar{z}_J)+\sum_{k}O(|z|^2)\frac{\partial v}{\partial z_k}+O(|z|)v.\]

We calculate \[[\bar{\partial}^*, L]u=\bar{\partial}^*(\omega\wedge u)-\omega\wedge\bar{\partial}^*u=-\frac{\partial}{\partial \bar{z}_k}\lrcorner\omega\wedge\frac{\partial u_{IJ}}{\partial z_k}(dz_I\wedge d\bar{z}_J)+O(|z|)=\i\frac{\partial u_{IJ}}{\partial z_k}dz_k\wedge(dz_I\wedge d\bar{z}_J)+O(|z|).\] We obtain that $[\bar{\partial}^*, L]u|_{x_0}=\i\partial u|_{x_0}$. Since $x_0\in\Omega$ is arbitrary, we obtain $[\bar{\partial}^*, L]=\i\partial$.

By taking conjugate and adjoint, we also obtain $[\partial^*, L]u=-\i\bar{\partial}$, $[\Lambda,\bar{\partial}]=-\i\partial^*$ and $[\Lambda, \partial]=\i\bar{\partial}^*$.

By the Jocobi's identity, we have the K\"{a}hler identity: \[\bar{\partial}\bar{\partial}^*+\bar{\partial}^*\bar{\partial}-\partial\partial^*-\partial^*\partial=[\bar{\partial},\bar{\partial}^*]-[\partial,\partial^*]=-\i[\bar{\partial},[\Lambda, \partial]]+\i[\partial,[\bar{\partial},\Lambda]]=-\i[\Lambda, [\partial,\bar{\partial}]]=0.\]

\subsection{The Bochner--Kodaira--Nakano identity}
To extend this result to a Hermitian holomorphic vector bundle, we consider a normal coordinate system. The following result is well-known, but for the sake of completeness, we give proof as follows.

\begin{proposition}[Demailly \cite{De12}, Page 270]
	Let $E\rightarrow X$ be a Hermitian holomorphic vector bundle. For every point $x_0\in X$ and every coordinate system $(z_j)_{1\leq j\leq n}$ at $x_0\in X$, there exists a holomorphic frame $(e_\lambda)_{1\leq\lambda\leq r}$ in a neighborhood of $x_0$ such that\[\inn{e_\lambda (z), e_\mu(z)}[\mathbb{C}]=\delta_{\lambda\mu}-\sum_{1\leq j,k\leq n}c_{jk\lambda\mu} z_j\bar{z}_k+O(|z|^3),\] where $(c_{jk\lambda\mu})$ are the coefficients of the Chern Curvature tensore $\Theta (E)|_{x_0}$, i.e., \[\Theta(E)|_{x_0} e_\lambda(x_0)=\sum_{j,k,\lambda,\mu} c_{jk\lambda\mu}dz_j\wedge d\bar{z}_k\otimes e_\mu(x_0).\]Such a frame $(e_\lambda)$ is called a normal coordinate frame at $x_0$.
\end{proposition}

\begin{proof}
	We fix $x_0\in X$ and take an arbitrary coordinate system $(h_j)_{1\leq j\leq n}$ so that $\inn{h_\lambda(x_0),h_\mu(x_0)}[\mathbb{C}]=\delta_{\lambda\mu}$. This can be done by at $x_0$ only. 
 
 By Taylor's theorem, we obtain that \[\inn{h_\lambda,h_\mu}[\mathbb{C}]=\delta_{\lambda\mu}+\sum_{j}a_{j\lambda\mu}z_j+\sum_{j}a'_{j\lambda\mu}\bar{z}_j+O(|z|^2).\] Since $\inn{h_\lambda,h_\mu}[\mathbb{C}]=\overline{\inn{h_\mu, h_\lambda}[\mathbb{C}]}$, we get $\bar{a}_{j\mu\lambda}=a'_{j\lambda\mu}$.
	
	For getting rid of the linear part in $\inn{h_\lambda,h_\mu}[\mathbb{C}]$, we make the first holomorphic coordinate change $g_\lambda(z):=h_\lambda(z)-\sum_{i,j}a_{j\lambda i}z_jh_i(z)$. Up to $O(|z|^2)$, we may compute \[\begin{split}
	    &\inn{g_\lambda, g_\mu}[\mathbb{C}]\\
     =&\inn{h_\lambda(z)-\sum_{i,j}a_{j\lambda i}z_jh_i(z),h_\mu(z)-\sum_{i,j}a_{j\mu i}z_jh_i(z)}[\mathbb{C}]\\
     =&\inn{h_\lambda(z),h_\mu(z)}[\mathbb{C}]-\sum_{i,j}\bar a_{j\mu i}\bar z_j\inn{h_\lambda(z),h_i(z)}[\mathbb{C}]-\sum_{i,j}a_{j\lambda i}z_j\inn{h_i(z),h_\mu(z)}[\mathbb{C}]+O(|z|^2)\\
     =&\delta_{\lambda\mu}+\sum_{j}a_{j\lambda\mu}z_j+\sum_{j}a'_{j\lambda\mu}\bar{z}_j-\sum_{j}\bar a_{j\mu \lambda}\bar z_j-\sum_{j}a_{j\lambda \mu}z_j+O(|z|^2)\\
     =&\delta_{\lambda\mu}+O(|z|^2).
	\end{split}\] When written up to $O(|z|^3)$, \[
	    \inn{g_\lambda, g_\mu}[\mathbb{C}]=\delta_{\lambda\mu}+\sum_{j,k}a_{jk\lambda\mu}z_j\bar{z}_k+\sum_{j,k}a'_{jk\lambda\mu}z_jz_k+\sum_{j,k}a''_{jk\lambda\mu}\bar{z}_j\bar{z}_k+O(|z|^3).
	\] Clearly, the following properties hold: $\bar{a}'_{jk\mu\lambda}=a''_{jk\lambda\mu}$ and $\bar{a}_{kj\mu\lambda}=a_{jk\lambda\mu}$.
	
	For simplifying the second order parts, we make the second holomorphic coordinate change $e_\lambda(z):=g_\lambda(z)-\sum_{j,k,i}a'_{jk\lambda i} z_jz_kg_i(z)$. Then
	\[\begin{split}
	    &\inn{e_\lambda,e_\mu}[\mathbb{C}]\\=&\inn{g_\lambda(z),g_\mu(z)}[\mathbb{C}]-\sum_{j,k,i}\bar a'_{jk\mu i} \bar z_j\bar z_k\inn{g_\lambda(z),g_i(z)}[\mathbb{C}]-\sum_{j,k,i}a'_{jk\lambda i} z_jz_k\inn{g_i(z),g_\mu(z)}[\mathbb{C}]+O(|z|^3)\\=&\delta_{\lambda\mu}+\sum_{j,k}a_{jk\lambda\mu}z_j\bar{z}_k+O(|z|^3).
	\end{split}\]
	
	Observe that,  \[d\inn{e_\lambda,e_\mu}[\mathbb{C}]=\sum_{j,k}a_{jk\lambda\mu}z_jd\bar{z}_k+\sum_{j,k}a_{jk\lambda\mu}\bar{z}_kdz_j+O(|z|^2).\] 
	
	The Chern connection respects $\inn{\cdot,\cdot}[\mathbb{C}]$. In other words, \[\sum_{j,k}a_{jk\lambda\mu}\bar{z}_kdz_j+O(|z|^2)=\inn{\partial_E e_\lambda, e_\mu}[\mathbb{C}]+\inn{e_\lambda, \bar\partial e_\mu}[\mathbb{C}].\] Since, $e_\mu$ is holomorphic, we obtain that \[\inn{\partial_E e_\lambda,e_\mu}[\mathbb{C}]=\sum_{j,k}a_{jk\lambda\mu}\bar{z}_kdz_j+O(|z|^2).\] This implies \[\partial_E e_\lambda=\sum_{j,k,\mu}a_{jk\lambda\mu}\bar{z}_k dz_j\otimes e_\mu+O(|z|^2)\] and so \[\Theta(E)e_\lambda=\bar{\partial}\partial_E e_\lambda=\sum_{j,k,\mu}a_{jk\lambda\mu}d\bar{z}_k\wedge dz_j\otimes e_\mu+O(|z|).\] This completes the proof when we define $c_{jk\lambda\mu}=-a_{jk\lambda\mu}$.
\end{proof}

Combining the geodesic coordinate system in $X$ and the normal coordinate system on the fiber of $E$, we are ready to prove the following Bochner--Kodaira--Nakano identity which generalizes the K\"{a}hler identity to Hermitian holomorphic vector bundle.

\begin{proposition}[Bochner--Kodaira--Nakano Identity, see Page 329 of Demailly \cite{De12}]
	Let $X$ be a K\"{a}hler manifold with metric $\omega$ and $E\rightarrow X$ be a Hermitian holomorphic vector bundle with fiber metric $h$. Let $\nabla=\partial_E+\bar{\partial}$ be the Chern connection. We have the identity: \[[\bar{\partial}^*_E, L]=\i\partial_E.\]
\end{proposition}

\begin{proof}
	On $X$, we fix a geodesic coordinate system $(z_i)_{i=1}^n$ and on $E$, we fix a normal coordinate system $(e_\lambda)_{\lambda=1}^r$. For any smooth section $s=\sum_\lambda \sigma_\lambda\otimes e_\lambda\in\mathscr{C}^\infty_{p,q}(X, E)$. We know, in a trivialization, 
	$\bar{\partial}_Es=\sum_\lambda \bar{\partial}\sigma_\lambda\otimes e_\lambda$ and $\partial_Es=\sum_\lambda \partial\sigma_\lambda\otimes e_\lambda+\sum_\lambda(-1)^{(p+q)}\sigma_\lambda \otimes\partial_E e_\lambda=\sum_\lambda \partial\sigma_\lambda\otimes e_\lambda+O(|z|)$.
	
	We are going to calculate $\bar{\partial}_E^*$. Let $t=\sum_\mu \tau_\mu\otimes e_\mu\in\mathscr{C}^\infty_{p,q+1;c}(U(x_0),E)$, $\bar{\partial}^*_Et=\sum_\mu \kappa_\mu\otimes e_\mu$ and a neighborhood $U(x_0)$ of $x_0$ in $X$. Consider 
	\[\int_{U(x_0)}\inn{\bar{\partial}_Es, t}[E]\,dV=\int_{U(x_0)}\inn{\sum_\lambda \bar{\partial}_E\sigma_\lambda\otimes e_\lambda, \sum_\mu \tau_\mu\otimes e_\mu}[E]\,dV=\sum_{\lambda, \mu}\int_{U(x_0)}\inn{\bar{\partial}_E\sigma_\lambda,  \tau_\mu}[\mathbb{C}]\inn{e_\lambda,e_\mu}[\mathbb{C}]\,dV\]
and 
\[\int_{U(x_0)}\inn{\bar{\partial}_Es, t}[E]\,dV=\int_{U(x_0)}\inn{s, \bar{\partial}^*_Et}[E]\,dV=\sum_{\lambda,\mu}\int_{U(x_0)}\inn{\sigma_\lambda,\kappa_\mu}[\mathbb{C}]\inn{e_\lambda,e_\mu}[\mathbb{C}]\,dV.\]

	Observe that \[\begin{split}
	    &\int_{U(x_0)}\inn{\bar{\partial}_Es, t}[E]\,dV=\sum_{\lambda, \mu}\int_{U(x_0)}\inn{\bar{\partial}\sigma_\lambda, \inn{e_\mu,e_\lambda}[\mathbb{C}]  \tau_\mu}[\mathbb{C}]\,dV=\sum_{\lambda, \mu}\int_{U(x_0)}\inn{\sigma_\lambda, \bar{\partial}^*\inn{e_\mu,e_\lambda}[\mathbb{C}]  \tau_\mu}[\mathbb{C}]\,dV\\=&\sum_{\lambda, \mu}\int_{U(x_0)}\inn{\sigma_\lambda, (\inn{e_\mu,e_\lambda}[\mathbb{C}])^{-1}\bar{\partial}^*\inn{e_\mu,e_\lambda}[\mathbb{C}]  \tau_\mu}[\mathbb{C}]\inn{e_\lambda,e_\mu}[\mathbb{C}]\,dV\\=&\sum_{\lambda, \mu}\int_{U(x_0)}\inn{\sigma_\lambda\otimes e_\lambda, \left(\inn{e_\mu,e_\lambda}[\mathbb{C}])^{-1}\bar{\partial}^*\inn{e_\mu,e_\lambda}[\mathbb{C}]  \tau_\mu\right)\otimes e_\mu}[E]\,dV,
	\end{split}\] and $\inn{e_\mu,e_\lambda}[\mathbb{C}]=\delta_{\mu\lambda}+O(|z|^2)$ implies $(\inn{e_\mu,e_\lambda}[\mathbb{C}])^{-1}=\delta_{\mu\lambda}+O(|z|^2)$.

 

	Consequently, we obtain that \[\bar{\partial}^*_E t=\sum_{\mu}\bar{\partial}^*_E (\tau_\mu\otimes e_\mu)=\sum_{\mu}(\bar{\partial}^* \tau_\mu)\otimes e_\mu+O(|z|).\]

 Now we verify the identity. Calculate \[\begin{split}
     &[\bar\partial^*_E, L]t=\sum_{\mu}[\bar\partial^*_E, L]\kappa_\mu\otimes e_\mu=\sum_{\mu}\bar\partial^*_E L\kappa_\mu\otimes e_\mu-\sum_{\mu} L\bar\partial^*_E\kappa_\mu\otimes e_\mu\\=&\sum_{\mu}(\bar\partial^* L\kappa_\mu)\otimes e_\mu-\sum_{\mu} (L\bar\partial^*\kappa_\mu)\otimes e_\mu+O(|z|)=\sum_\mu[\bar\partial^*, L]\kappa_\mu\otimes e_\mu+O(|z|)=\i\sum_\mu\partial\kappa_\mu\otimes e_\mu+O(|z|).
 \end{split}\]

 At $x_0$, we find that $[\bar\partial^*_E, L]=\i\partial_E$ which completes the proof.
\end{proof}

 \subsection{Comparison of Laplacians}\label{identityofLap}
 Let $\Box_E=\bar\partial_E\bar\partial_E^*+\bar\partial_E^*\bar\partial_E=[\bar\partial_E,\bar\partial_E^*]$ and $\bar\Box_E=\partial_E\partial_E^*+\partial_E^*\partial_E=[\partial_E,\partial_E^*]$, where the $[\cdot,\cdot]$ denotes the supercommutator, i.e., $[A, B]=AB-(-1)^{\mathrm{deg}(A)\mathrm{deg}(B)}BA$.
\begin{proposition}\label{diffbetweenLap}
    Let $X$ be K\"{a}hler. \[\Box_E-\bar\Box_E=\i[\Theta_E, \Lambda].\]
\end{proposition}
\begin{proof}
    Observe that \[\Box_E=[\bar\partial,\bar\partial_E^*]=-\i[\bar\partial, [\Lambda, \partial_E]]=-\i[\Lambda, [\partial_E,\bar\partial]]-\i[\partial_E, [\bar\partial, \Lambda]]\] and \[\bar\Box_E=[\partial_E,\partial_E^*]=\i[\partial_E,[\Lambda, \bar{\partial}]]=-\i[\partial_E,[\bar{\partial}, \Lambda]].\]

    Consequently, we obtain that \[\Box_E-\bar\Box_E=-\i[\Lambda, [\partial_E,\bar\partial]]=-\i[\Lambda, \Theta_E]=\i[\Theta_E, \Lambda].\]
\end{proof}

Let $E:=X\times\mathbb{C}\rightarrow X$ be a trivial line bundle with fiber metric $e^{-\phi}$. Proposition \ref{diffbetweenLap} gives pointwise, \[\inn{\Box_E u, u}e^{-\phi}-\inn{\bar\Box_Eu, u}e^{-\phi}=\inn{\i[\Theta_E, \Lambda]u,u}e^{-\phi},\] for all $u\in \mathscr{C}_{p,q;c}^\infty(X, \mathbb{C})$.
Integrating the equation above, we obtain that \[\int\inn{\Box_E u, u}e^{-\phi}\,dV-\int\inn{\bar\Box_Eu, u}e^{-\phi}\,dV=\int\inn{\i[\Theta_E, \Lambda]u,u}e^{-\phi}\,dV.\]

For the following, we introduce a special case on which the matters may be simplified. This is motivated by Ma--Marinescu \cite{MM07}. In this case, we only consider $\mathscr{C}_{0,q}^\infty(X, \mathbb{C})$. We can identify $\mathscr{C}_{0,q}^\infty(X, \mathbb{C})$ with $\mathscr{C}_{n,q}^\infty(X, K_X^*)$ by sending $u\in \mathscr{C}_{0,q}^\infty(X, \mathbb{C})$ to $\tilde u:=\omega^1\wedge\cdots\wedge\omega^n u\otimes K_X^*$, where $K_X^*:=\omega_1\wedge\cdots\wedge\omega_n=\det(T^{(1,0)}X)$ is the canonical bundle.

Applying the equation in Proposition \ref{diffbetweenLap} to $\tilde u\in \mathscr{C}_{n,q;c}^\infty(X, K_X^*)$, we obtain that \[\Box_{\tilde E}\tilde u-\bar\Box_{\tilde E}\tilde u=\i[\Theta_{\tilde E}, \Lambda]\tilde u.\] Since $\tilde u\in\mathscr{C}_{n,q;c}^\infty(X, K_X^*)$, we simplify the equation above to $\Box_{\tilde E}\tilde u-\partial_{\tilde E}\partial_{\tilde E}^* \tilde u=\i\Theta_{\tilde E} \Lambda\tilde u$. Consequently, \[\inn{\Box_{\tilde E}\tilde u, \tilde u}-\inn{\partial_{\tilde E}\partial_{\tilde E}^* \tilde u,\tilde u}=\i\inn{\Theta_{\tilde E} \Lambda\tilde u,\tilde u}.\]

When $X$ is a bounded domain with smooth boundary in $\mathbb{C}^n$, the K\"{a}hler metric is $\delta_{ij}$ and $K_X^*=\mathbb{C}$. We obtain that \[\inn{(\bar\partial\bar\partial_E^*+\bar\partial_E^*\bar\partial) u, u}-\inn{\partial_{E}\partial_{E}^* \tilde u,\tilde u}=\i\inn{\Theta_{E} \Lambda \tilde u, \tilde u}.\]

This results in the Morrey-Kohn-H\"{o}rmander's formula.

\begin{theorem}[Morrey-Kohn-H\"{o}rmander's formula]
  Let $\Omega$ be a bounded pseudoconvex domain with smooth boundary in $\mathbb{C}^n$. Let $\phi\in C^2(\overline{\Omega})$ be a function. Then we have that, for arbitrary $u\in\Domain(\bar\partial^*)\cap C^\infty_{(0,q)}(\overline{\Omega})$,
  \[\|\partial^*\tilde u\|_\phi^2=\|\bar\partial u\|_\phi^2+\|\bar\partial^*_\phi u\|^2_\phi-\dinn{\i\Theta_{\phi}\Lambda \tilde u, \tilde u}_\phi-\int_{\partial\Omega}e^{-\phi}\Hessian_\delta(u,u)\,d\sigma\]
\end{theorem}

\section{Estimates involving two equations}\label{sectionof2eq}

\subsection{Computation involving different weights}

We are going to calculate various operators under different weights. Let $X\times\mathbb{C}\rightarrow X$ be a trivial bundle, where $X$ is a K\"{a}hler manifold with metric $\omega$ (either complete or incomplete). Let $\phi\in\mathscr{C}^\infty_{0,0}(X,\mathbb{R})$. And let the $e^{-\phi}$ be a fiber metric on $X\times\mathbb{C}$. The curvature of the bundle can be computed as \[\i\Theta(X\times\mathbb{C})=\i\partial\bar{\partial}\phi.\]

The operator $\partial_E$ is defined to be \[\partial_E u=e^\phi\partial(e^{-\phi}u)=\partial u-\partial\phi\wedge u.\] 

Under the fiber metric $e^{-\phi}$, we have that $\|u\|^2_{\phi}=\int_{X}|u|^2e^{-\phi}\,dV$ and \[\|\partial_E u\|^2_{\phi}=\int_{X}|e^\phi\partial e^{-\phi}u|^2e^{-\phi}\,dV=\int_{X}|\partial e^{-\phi}u|^2e^{\phi}\,dV=\int_{X}|-\partial \phi\wedge u+\partial u|^2e^{-\phi}\,dV.\]

On the other hand, we calculate \[\|\partial_E u\|^2_{\phi}=\int_{X}|e^{\frac{\phi}{2}}\partial e^{-\frac{\phi}{2}}e^{-\frac{\phi}{2}}u|^2\,dV.\] Let $v=e^{-\frac{\phi}{2}}u$, and we get \[\|\partial_E u\|^2_{\phi}=\int_{X}|e^{\frac{\phi}{2}}\partial e^{-\frac{\phi}{2}}v|^2\,dV=\int_{X}|-\frac{1}{2}\partial\phi\wedge v+\partial v|^2\,dV.\]

The $\partial^*_E$ is independent of $\phi$ (because the weight has been taken care of by $\partial_E$.) and so we will write $\partial^*$ instead of $\partial^*_E$. Note \[\|\partial_E^*u\|_\phi=\int_{X}|\partial^* u|^2e^{-\phi}\,dV=\int_{X}|\partial^* e^{\frac{\phi}{2}}v|^2e^{-\phi}\,dV.\]

We compute $\partial^*$ in geodesic coordinate, letting $t=\sum_{K,L}t_{KL}dz_K\wedge d\bar{z}_L$: \[\partial^* t=-\sum_{K,L,k}\frac{\partial t_{KL}}{\partial \bar{z}_k}\frac{\partial}{\partial z_k}\lrcorner dz_K\wedge d\bar{z}_L+O(|z|).\] We observe that \[\begin{split}
	&\partial^*(e^\phi t)=-\sum_{K,L,k}\frac{\partial e^\phi t_{KL}}{\partial \bar{z}_k}\frac{\partial}{\partial z_k}\lrcorner dz_K\wedge d\bar{z}_L+O(|z|)\\=&-\sum_{K,L,k}\left(e^\phi\frac{\partial  t_{KL}}{\partial \bar{z}_k}+t_{KL}e^\phi \frac{\partial \phi }{\partial \bar{z}_k}\right)\frac{\partial}{\partial z_k}\lrcorner dz_K\wedge d\bar{z}_L+O(|z|)\\
	=&e^\phi\partial^*t-e^\phi\sum_{K,L,k}t_{KL} \frac{\partial \phi }{\partial \bar{z}_k}\frac{\partial}{\partial z_k}\lrcorner dz_K\wedge d\bar{z}_L+O(|z|).
\end{split}\]

Compute $\bar{\partial}^*$ in geodesic coordinates (see Demailly \cite{De12} Page 305). Observe that letting $s=\sum_{I,J} s_{IJ}dz_I\wedge d\bar{z}_J$ and $t=\sum_{K,L}t_{KL}dz_K\wedge d\bar{z}_L$, \[\int_{U(x_0)}\inn{\bar{\partial}s, t}[\mathbb{C}]\,dV=\sum_{I,J,k,K,L}\int_{U(x_0)}\inn{\frac{\partial s_{IJ}}{\partial\bar{z}_i}d\bar{z}_i\wedge dz_I\wedge d\bar{z}_J, t_{KL}dz_K\wedge d\bar{z}_L}_\mathbb{C}(1+O(|z|^2))\,dV_0.\] Consequently, \[\bar{\partial}^* t=-\sum_{K,L,k}\frac{\partial t_{KL}}{\partial z_k}\frac{\partial}{\partial \bar{z}_k}\lrcorner dz_K\wedge d\bar{z}_L+O(|z|).\]

\begin{definition}
	Assume $(dz_i)$ is an orthonormal basis at $x_0\in X$. Let $u=\sum_{i}u_{i}dz_i\in (T^{(1,0)}X)^*$, we define $u^\sharp=\sum_{i}u_{i}\frac{\partial}{\partial \bar{z}_i}\in T^{(0,1)}X$. 
\end{definition}
The definition enjoys the property: $u(Y)=\inn{u^\sharp, \overline{Y}}[\mathbb{C}]$ and $\inn{\alpha, \gamma^\sharp \lrcorner \beta}[\mathbb{C}]=\inn{\bar{\gamma}\wedge\alpha,\beta}[\mathbb{C}]$. The second equation is proved below.
\begin{lemma}
 \[\inn{\alpha, \gamma^\sharp \lrcorner \beta}[\mathbb{C}]=\inn{\bar{\gamma}\wedge\alpha,\beta}[\mathbb{C}].\]
\end{lemma}
\begin{proof}
	Choose the geodesic coordinates $(z_i)$ at $x_0\in X$. Let $\alpha=\sum_{I,J}\alpha_{IJ}dz_I\wedge d\bar{z}_J$, $\beta=\sum_{K,L}\beta_{KL}dz_K\wedge d\bar{z}_L$ and $\gamma=\sum_i\gamma_idz_i$. At $x_0$, we have that  $\gamma^\sharp=\sum_i\gamma_i\frac{\partial}{\partial \bar{z}_i}$. We observe that $\inn{\alpha, \gamma^\sharp \lrcorner \beta}[\mathbb{C}]=\sum_{i,I,i\in J}\alpha_{I(J\backslash\lbrace i\rbrace)}\bar{\gamma}_i\bar{\beta}_{IJ}$. Denote $J\backslash\lbrace i\rbrace$ by $J'$, we shift index and obtain that $\inn{\alpha, \gamma^\sharp \lrcorner \beta}[\mathbb{C}]=\sum_{i,I,i\notin J'}\alpha_{IJ'}\bar{\gamma}_i\bar{\beta}_{I(iJ')}$. Since $\inn{\bar{\gamma}\wedge\alpha,\beta}[\mathbb{C}]=\sum_{i,I,i\notin J}\bar\gamma_i\alpha_{IJ}\bar{\beta}_{I(iJ)}$, we complete the proof.
\end{proof}

\begin{remark}
Indeed, the $\sharp$ operator depends on the K\"{a}hler metric.
\end{remark}

With this definition, we have the following lemma.

\begin{lemma}\label{calcwithmultiplier}
	\[\bar{\partial}^* e^\phi t=e^\phi\left(\bar{\partial}^*t-(\partial\phi)^\sharp\lrcorner t\right)\] Similarly, \[\partial^* e^\phi t=e^\phi\left(\partial^*t-(\bar\partial\phi)^\sharp\lrcorner t\right)\]
\end{lemma}

\begin{proof}
	Take geodesic coordinates $(z_i)$ at $x_0\in X$. Observe that \[\begin{split}
		&\bar{\partial}^* e^\phi t\\=&-\sum_{K,L,k}\frac{\partial e^\phi t_{KL}}{\partial z_k}\frac{\partial}{\partial \bar{z}_k}\lrcorner dz_K\wedge d\bar{z}_L+O(|z|)\\
		=&-\sum_{K,L,k} e^\phi \frac{\partial t_{KL}}{\partial z_k}\frac{\partial}{\partial \bar{z}_k}\lrcorner dz_K\wedge d\bar{z}_L- \sum_{K,L,k}e^\phi t_{KL}\frac{\partial\phi}{\partial z_k}\frac{\partial}{\partial \bar{z}_k}\lrcorner dz_K\wedge d\bar{z}_L+O(|z|).
	\end{split}\]

Restrict at $x_0$, we obtain that \[\bar{\partial}^* e^\phi t=e^\phi\left(\bar{\partial}^*t-(\partial\phi)^\sharp\lrcorner t\right).\] Since $x_0\in X$ is arbitrary, the equation is proved.
\end{proof}

Consequently, we have that \[e^\phi\bar{\partial}^* e^{-\phi} u=\bar{\partial}^*u+(\partial\phi)^\sharp\lrcorner u\] and letting $v=e^{-\frac{\phi}{2}}u$, \[|e^\phi\bar{\partial}^* e^{-\phi} u|^2_\mathbb{C}e^{-\phi}=|\bar{\partial}^*u+(\partial\phi)^\sharp\lrcorner u|^2_\mathbb{C}e^{-\phi}=|\bar{\partial}^*v+\frac{1}{2}(\partial\phi)^\sharp\lrcorner v|^2_\mathbb{C}.\]

\subsection{The case without boundary (complete metric on $X$)}

In this section, the $\Domain(\bar\partial*)$ and $\Domain(\bar\partial)$ are defined as in the book of Chen--Shaw \cite{CS01} and Straube \cite{St10}.

By the Bochner--Kodaira--Nakano identity, we compute pointwise \begin{equation}\label{11}
    \bar\partial\bar\partial^*_\phi u+\bar\partial^*_\phi\bar\partial u=\partial_\phi\partial^*u+\partial^*\partial_\phi u+\i[\Theta_{\phi},\Lambda] u
\end{equation} and \[\bar\partial\bar\partial^*_{\phi+\eta} u+\bar\partial^*_{\phi+\eta}\bar\partial u=\partial_{\phi+\eta}\partial^*u+\partial^*\partial_{\phi+\eta}u+\i[\Theta_{\phi+\eta},\Lambda] u.\] The second equation can be simplified to \[
    \bar\partial e^\eta\bar\partial^*_{\phi} e^{-\eta}u+e^\eta\bar\partial^*_{\phi}e^{-\eta}\bar\partial u=e^\eta\partial_{\phi}e^{-\eta}\partial^*u+\partial ^*e^\eta\partial_\phi e^{-\eta}u+\i[\Theta_{\phi+\eta},\Lambda] u\] or equivalently, \begin{equation}\label{22}
        \bar\partial (\bar\partial^*_\phi u+(\partial\eta)^\sharp\lrcorner u)+\bar\partial^*_\phi\bar\partial u+(\partial\eta)^\sharp\lrcorner \bar\partial u=\partial_\phi\partial^*u-\partial\eta\wedge\partial^* u+\partial ^*(\partial_\phi u-\partial
    \eta\wedge u)+\i[\Theta_{\phi+\eta},\Lambda] u.
    \end{equation}

Consequently, taking a difference of the two equations (\ref{11}) and (\ref{22}), \[\bar\partial(\partial\eta)^\sharp\lrcorner u+(\partial\eta)^\sharp\lrcorner \bar\partial u=-\partial\eta\wedge\partial^* u-\partial^*\partial\eta\wedge u+\i[\Theta_{\eta},\Lambda] u.\]

Then we obtain the so-called basic equation in the note:

\begin{proposition}[The basic equation]\label{basicequation}
    \[\inn{\bar\partial(\partial\eta)^\sharp\lrcorner u, u}_\mathbb{C}e^{-\phi}+\inn{(\partial\eta)^\sharp\lrcorner \bar\partial u,u}_\mathbb{C}e^{-\phi}=-\inn{\partial\eta\wedge\partial^* u,u}_\mathbb{C}e^{-\phi}-\inn{\partial^*\partial
    \eta\wedge u,u}[\mathbb{C}]e^{-\phi}+\i\inn{[\Theta_{\eta},\Lambda] u,u}_\mathbb{C}e^{-\phi}.\]
\end{proposition}

\begin{lemma}\label{sharpwedge}
	\[|(\partial\eta)^\sharp\lrcorner v|^2_\mathbb{C}=-|\bar{\partial}\eta\wedge v|^2_\mathbb{C}+|\bar\partial\eta|^2_\mathbb{C}|v|^2_\mathbb{C}.\]
\end{lemma}
\begin{proof}
		Observe that \[\inn{(\partial\eta)^\sharp\lrcorner v, (\partial\eta)^\sharp\lrcorner v}[\mathbb{C}]=\inn{\bar{\partial}\eta\wedge (\partial\eta)^\sharp\lrcorner v, v}[\mathbb{C}].\] By the Leibniz rule: \[\bar{\partial}\eta\wedge (\partial\eta)^\sharp\lrcorner v=-(\partial\eta)^\sharp\lrcorner \bar{\partial}\eta\wedge v+\left((\partial\eta)^\sharp\lrcorner\bar{\partial}\eta\right)\wedge  v,\] we have that  \[
			\inn{(\partial\eta)^\sharp\lrcorner v, (\partial\eta)^\sharp\lrcorner v}[\mathbb{C}]=\inn{-(\partial\eta)^\sharp\lrcorner \bar{\partial}\eta\wedge v+\left((\partial\eta)^\sharp\lrcorner\bar{\partial}\eta\right)\wedge  v, v}[\mathbb{C}]=-|\bar{\partial}\eta\wedge v|^2_\mathbb{C}+|\bar\partial\eta|^2_\mathbb{C}|v|^2_\mathbb{C}.
		\]
\end{proof}


Note that the Bochner--Kodaira--Nakano identity gives the following well-known equation:
\begin{proposition}\label{BKNnob}
    Let $E:=X\times\mathbb{C}\rightarrow X$ be a trivial line bundle with fiber metric $e^{-\phi}$. 
We have that \[\|\bar{\partial} u\|_\phi+\|\bar{\partial}^*_\phi u\|_\phi-\|\partial_\phi u\|_\phi-\|\partial^* u\|_\phi=\dinn{\i[\Theta_\phi,\Lambda]u,u}[\phi],\] for all $u\in \mathscr{C}^\infty_{p,q;c}(X, \mathbb{C})$ 
\end{proposition}

We will imporve the proposition above by the following proposition. 

\begin{proposition}
Let $X\times\mathbb{C}\rightarrow X$ be a trivial line bundle with fiber metric $e^{-\phi}$. If there exists $t>0$ such that $\i[\Theta_{\eta+t\phi}, \Lambda]-t^{-1}|\bar{\partial}\eta|^2_\mathbb{C}$ positive definite, we have the following \textit{a priori} estimate \[C\|\bar{\partial}v\|_\phi^2+C\|\bar{\partial}^*_\phi v\|^2_\phi\geq\|v\|_\phi^2,\] for all $v\in \mathscr{C}^\infty_{p,q;c}(X, \mathbb{C})$.
\end{proposition}

\begin{proof}
Recall that from Proposition \ref{basicequation} \[\inn{\bar\partial(\partial\eta)^\sharp\lrcorner v, v}_\mathbb{C}e^{-\phi}+\inn{(\partial\eta)^\sharp\lrcorner \bar\partial v,v}_\mathbb{C}e^{-\phi}=-\inn{\partial\eta\wedge\partial^* v,v}_\mathbb{C}e^{-\phi}-\inn{\partial^*\partial
    \eta\wedge v,v}[\mathbb{C}]e^{-\phi}+\i\inn{[\Theta_{\eta},\Lambda] v,v}_\mathbb{C}e^{-\phi}.\]

Consequently,
\[\begin{split}
	&\int\inn{\bar{\partial}v, \bar{\partial}\eta\wedge v}[\mathbb{C}]e^{-\phi}\,dV+\int\inn{(\partial\eta)^\sharp\lrcorner v, \bar{\partial}^*_\phi v}[\mathbb{C}]e^{-\phi}\,dV-\i\int\inn{[\Theta_\eta,\Lambda]v,v}[\mathbb{C}]e^{-\phi}\,dV\\=&-\int\inn{\partial\eta\wedge v, \partial_\phi v}[\mathbb{C}]e^{-\phi}\,dV-\int\inn{\partial^*v, (\bar{\partial}\eta)^\sharp\lrcorner v}[\mathbb{C}] e^{-\phi}\,dV.\end{split}\]

This implies,
\[\begin{split}
	&\Re\int\inn{\bar{\partial}v, \bar{\partial}\eta\wedge v}[\mathbb{C}]e^{-\phi}\,dV+\Re\int\inn{\bar{\partial}^*_\phi v, (\partial\eta)^\sharp\lrcorner v}[\mathbb{C}]e^{-\phi}\,dV-\i\int\inn{[\Theta_\eta,\Lambda]v,v}[\mathbb{C}]e^{-\phi}\,dV\\\geq&-t\|\partial_\phi v\|_\phi^2-\frac{1}{2t}\|\partial\eta\wedge v\|_\phi^2-t\|\partial^*v\|_\phi^2-\frac{1}{2t}\|(\bar{\partial}\eta)^\sharp\lrcorner v\|_\phi^2\\
	=&-t\|\bar{\partial} v\|^2_\phi-t\|\bar{\partial}^*_\phi v\|^2_\phi+\i\int\inn{[\Theta_{t\phi},\Lambda]v,v}[\mathbb{C}]e^{-\phi}\,dV-\frac{1}{2t}\|\partial\eta\wedge v\|_\phi^2-\frac{1}{2t}\|(\bar{\partial}\eta)^\sharp\lrcorner v\|_\phi^2.\end{split}\]

 The last equation is due to Proposition \ref{BKNnob}.

The last inequality gives that, \[\begin{split}
    &2t\|\bar{\partial}v\|_\phi^2+2t\|\bar{\partial}^*_\phi v\|^2_\phi\\\geq&\i\int\inn{[\Theta_{\eta+t\phi},\Lambda]v,v}[\mathbb{C}]e^{-\phi}\,dV-\frac{1}{2t}\|\partial\eta\wedge v\|^2_\phi-\frac{1}{2t}\|(\bar\partial\eta)^\sharp\lrcorner v\|^2_\phi-\frac{1}{2t}\|\bar{\partial}\eta\wedge v\|^2_\phi-\frac{1}{2t}\|(\partial\eta)^\sharp\lrcorner v\|^2_\phi,
\end{split}
	\] where $\Theta_{\eta+t\phi}$ denotes the curvature form of the fiber metric $|\cdot|^2_{\eta+t\phi}=|\cdot|^2_\mathbb{C}e^{-\eta-t\phi}$.

	Recall that \[|(\partial\eta)^\sharp\lrcorner v|^2_\mathbb{C}=-|\bar{\partial}\eta\wedge v|^2_\mathbb{C}+|\bar\partial\eta|^2_\mathbb{C}|v|^2_\mathbb{C}.\] Consequently, \[\int \left(|\bar{\partial}\eta\wedge v|^2_\mathbb{C}+|(\partial\eta)^\sharp\lrcorner v|^2_\mathbb{C}\right)e^{-\phi}\,dV=\int |\bar\partial\eta|^2_\mathbb{C}|v|^2_\mathbb{C}e^{-\phi}\,dV=\||\bar{\partial}\eta|_\mathbb{C}v\|_\phi\] and thus
\[
t\|\bar{\partial}v\|_\phi^2+t\|\bar{\partial}^*_\phi v\|^2_\phi\geq\i\int\inn{[\Theta_{\eta+t\phi},\Lambda]v,v}[\mathbb{C}]e^{-\phi}\,dV-\frac{1}{t}\int \inn{|\bar\partial\eta|^2_\mathbb{C}v,v}[\mathbb{C}]e^{-\phi}\,dV.\]

Consequently, \[
\|\bar{\partial}v\|_\phi^2+\|\bar{\partial}^*_\phi v\|^2_\phi\geq\i t^{-1}\int\inn{[\Theta_{\eta+t\phi},\Lambda]v,v}[\mathbb{C}]e^{-\phi}\,dV-t^{-2}\int \inn{|\bar\partial\eta|^2_\mathbb{C}v,v}[\mathbb{C}]e^{-\phi}\,dV.\]

If there exists $t>0$ such that $\i[\Theta_{\eta+t\phi}, \Lambda]-t^{-1}|\bar{\partial}\eta|^2_\mathbb{C}$ positive definite, then the conclusion follows.

\end{proof}

\begin{remark}
	In the preceding proposition, we did not assume positivity on the curvature of $X\times\mathbb{C}$.
\end{remark}

The preceding proposition can be extend for all $v\in \Domain(\bar{\partial})\cap\Domain(\bar{\partial}^*)$ through the following density lemma.

\begin{lemma}[Theorem 2.6 in Ohsawa \cite{Oh18}]\label{densitylemma}
	The $\mathscr{C}_{p, q;_c}(X,\mathbb{C})\subset L^2_{p,q}(X, \mathbb{C})$ is dense in $\Domain(\bar{\partial})\cap\Domain(\bar{\partial}^*)$ with respect to the graph norm $\|u\|^2+\|\bar{\partial}u\|^2+\|\bar{\partial}^*u\|^2$
\end{lemma}

From now on, we define $A_{\eta, t}=\frac{\i}{2}t^{-1}[\Theta_{\eta+t\phi}, \Lambda]-\frac{t^{-2}}{2}|\bar{\partial}\eta|^2_\mathbb{C}$. By Lemma \ref{densitylemma} and the above inequality, we obtain that \[
\|\bar{\partial}v\|_\phi^2+\|\bar{\partial}^*_\phi v\|^2_\phi\geq \dinn{A_{\eta, t}v, v}[\phi]\] holds for all $v\in\Domain(\bar{\partial})\cap\Domain(\bar{\partial}^*)$.

Suppose we have $\bar{\partial}u=f$ for a given $f\in L^2_{n, q+1}(X, E)$ such that $\bar{\partial}f=0$. We consider for arbitrary  $s\in\Domain(\bar{\partial})\cap\Domain(\bar{\partial}^*)$, that $\dinn{f, s}[\phi]\leq \dinn{A_{\eta, t}s, s}[\phi]\dinn{A^{-1}_{\eta, t}f,f}[\phi]$ by the Cauchy--Schwartz inequality. Consequently, \[\dinn{f, s}[\phi]\leq \dinn{A^{-1}_{\eta, t}f,f}[\phi]\left(\|\bar{\partial}s\|^2_\phi+\|\bar{\partial}^*_\phi s\|^2_\phi\right).\] We decompose $s=s_1+s_2$, where $s_1\in \ker(\bar{\partial})$ and $s_2\in \overline{\Range (\bar{\partial}^*)}$. Since $f\in\ker(\bar{\partial})$, we have that \[\dinn{f,s}[\phi]=\dinn{f,s_1}[\phi]\leq \dinn{A^{-1}_{\eta, t}f,f}[\phi]\|\bar{\partial}^*_\phi s_1\|^2_\phi.\] Since $s_2\in\overline{\Range(\bar{\partial}^*)}\subset\ker(\bar{\partial}^*)$. So $\dinn{f,s}[\phi]\leq \dinn{A^{-1}_{\eta, t}f,f}[\phi]\|\bar{\partial}^*_\phi s\|^2_\phi$. By the Hahn--Banach theorem, the well-defined functional on $\Range(\bar{\partial}^*)$: $\bar{\partial}^*s\mapsto \dinn{f,s}[\phi]$ extends to a bounded linear functional on $L^2_{p, q-1}(X, E)$. Consequently, there exists $u\in L^2_{p, q-1}(X, E)$ with $\|u\|_\phi\leq \dinn{A^{-1}_{\eta, t}f,f}[\phi]$ and $u$ solves the equation $\bar{\partial}u=f$ in $L^2_{p, q-1}(X, E)$. This summarizes in the following theorem.

\begin{theorem}
	Let $X\times\mathbb{C}\rightarrow X$ be a trivial line bundle with fiber metric $e^{-\phi}$. Suppose there exists $t>0$ such that $\i[\Theta_{\eta+t\phi}, \Lambda]-t^{-1}|\bar{\partial}\eta|^2_\mathbb{C}$ positive definite. For all $f\in L^2_{p, q}(X, E)$ satisfying $\dinn{A^{-1}_{\eta, t}f,f}[\phi]<\infty$, there exists a $u\in L^2_{p, q-1}(X, E)$ solves $\bar{\partial}u=f$. Moreover, we have the estimate \[\|u\|_\phi\leq \dinn{A^{-1}_{\eta, t}f,f}[\phi].\]
\end{theorem}

\subsection{The case with boundary (incomplete metric on $X$)}

By the Bochner--Kodaira--Nakano identity, we have that, for $u\in C^\infty_{(n,q)}(\overline{\Omega})$, \[\bar\partial\bar\partial^*_\phi u+\bar\partial^*_\phi\bar\partial u=\partial_\phi\partial^*u+\i\Theta_{\phi}\Lambda u.\] The second equation is that  \[\bar\partial\bar\partial^*_{\phi+\eta} u+\bar\partial^*_{\phi+\eta}\bar\partial u=\partial_{\phi+\eta}\partial^*u+\i\Theta_{\phi+\eta}\Lambda u\] or equivalently, \[
    \bar\partial e^\eta\bar\partial^*_{\phi} e^{-\eta}u+e^\eta\bar\partial^*_{\phi}e^{-\eta}\bar\partial u=e^\eta\partial_{\phi}e^{-\eta}\partial^*u+\i\Theta_{\phi+\eta}\Lambda u\]which is,\[\bar\partial (\bar\partial^*_\phi u+(\partial\eta)^\sharp\lrcorner u)+\bar\partial^*_\phi\bar\partial u+(\partial\eta)^\sharp\lrcorner \bar\partial u=\partial_\phi\partial^*u-\partial\eta\wedge\partial^* u+\i\Theta_{\phi+\eta}\Lambda u.
\]

Consequently, taking difference of the two equations \[\bar\partial(\partial\eta)^\sharp\lrcorner u+(\partial\eta)^\sharp\lrcorner \bar\partial u=-\partial\eta\wedge\partial^* u+\i\Theta_{\eta}\Lambda u.\]

Then the basic equation:

\begin{proposition}[The basic equation for incomplete metric case]\label{basicequationincomplete}
    For $u\in C^\infty_{(n,q)}(\overline{\Omega})$, \[\inn{\bar\partial(\partial\eta)^\sharp\lrcorner u, u}_\mathbb{C}e^{-\phi}+\inn{(\partial\eta)^\sharp\lrcorner \bar\partial u,u}_\mathbb{C}e^{-\phi}=-\inn{\partial\eta\wedge\partial^* u,u}_\mathbb{C}e^{-\phi}+\i\inn{\Theta_{\eta}\Lambda u,u}_\mathbb{C}e^{-\phi}.\]
    In particular, if $u\in C^\infty_{(n,q)}(\overline{\Omega})\cap\ker(\bar\partial)$, we have that \[\inn{\bar\partial(\partial\eta)^\sharp\lrcorner u, u}_\mathbb{C}e^{-\phi}=-\inn{\partial\eta\wedge\partial^* u,u}_\mathbb{C}e^{-\phi}+\i\inn{\Theta_{\eta}\Lambda u,u}_\mathbb{C}e^{-\phi}.\]
\end{proposition}

Then furthermore, if $u\in\ker(\bar\partial)\cap\Domain(\bar\partial^*)\cap C^\infty_{(n,q)}(\overline{\Omega})$, \[\dinn{(\partial\eta)^\sharp\lrcorner u, \bar\partial^*_\phi u}_\phi+\dinn{\bar\partial u, \bar\partial\eta\wedge u}_\phi=-\dinn{\partial^* u,(\bar\partial\eta)^\sharp\lrcorner u}_\phi+\dinn{\i\Theta_{\eta}\Lambda u,u}_\phi.\] 

Observe that, by Morrey-Kohn-H\"{o}rmander's formula \[\begin{split}
    &\Re\dinn{\partial^* u,(\bar\partial\eta)^\sharp\lrcorner u}_\phi\\\geq&-t_3\|\partial^*u\|^2_\phi-\frac{1}{4t_3}\|(\bar\partial\eta)^\sharp\lrcorner u\|_\phi\\=&-t_3\left(\|\bar\partial u\|_\phi^2+\|\bar\partial^*_\phi u\|^2_\phi-\dinn{\i\Theta_{\phi}\Lambda u, u}_\phi-\int_{\partial\Omega}e^{-\phi}\Hessian_\delta(u,u)\,d\sigma\right)-\frac{1}{4t_3}\|(\bar\partial\eta)^\sharp\lrcorner u\|_\phi^2
\end{split}\]
Consequently, we have the following result:
\begin{theorem}
    For $u\in C^\infty_{(n,q)}(\overline{\Omega})\cap\Domain(\bar\partial^*)$, \[
    (t_1+t_3)\|\bar\partial u\|_\phi^2+(t_2+t_3)\|\bar\partial^*_\phi u\|^2_\phi\geq-\frac{1}{4t_3}\|(\bar\partial\eta)^\sharp\lrcorner u\|_\phi^2+\dinn{\i\Theta_{t_3\phi+\eta}\Lambda u,u}_\phi-\frac{1}{4t_2}\|(\partial\eta)^\sharp\lrcorner u\|^2_\phi-\frac{1}{4t_1}\|\bar\partial\eta\wedge u\|^2_\phi.
\]
In particular, if $u\in\Domain(\bar\partial^*)\cap\ker(\bar\partial)\cap C^\infty_{(n,q)}(\overline{\Omega})$, \[
    (t_2+t_3)\|\bar\partial^*_\phi u\|^2_\phi\geq-\frac{1}{4t_3}\|(\bar\partial\eta)^\sharp\lrcorner u\|_\phi^2+\dinn{\i\Theta_{t_3\phi+\eta}\Lambda u,u}_\phi-\frac{1}{4t_2}\|(\partial\eta)^\sharp\lrcorner u\|^2_\phi.\]
\end{theorem}
\begin{proof}

By Proposition \ref{basicequationincomplete} and Morrey--Kohn--H\"{o}rmander's formula,
    \[\begin{split}
    &(t_1+t_3)\|\bar\partial u\|_\phi^2+(t_2+t_3)\|\bar\partial^*_\phi u\|^2_\phi+\frac{1}{4t_2}\|(\partial\eta)^\sharp\lrcorner u\|^2_\phi+\frac{1}{4t_1}\|\bar\partial\eta\wedge u\|^2_\phi\\\geq &t_3\dinn{\i\Theta_{\phi}\Lambda u, u}_\phi-\frac{1}{4t_3}\|(\bar\partial\eta)^\sharp\lrcorner u\|_\phi^2+\dinn{\i\Theta_{\eta}\Lambda u,u}_\phi\\=&-\frac{1}{4t_3}\|(\bar\partial\eta)^\sharp\lrcorner u\|_\phi^2+\dinn{\i\Theta_{t_3\phi+\eta}\Lambda u,u}_\phi.
\end{split}\]

\end{proof}

\begin{lemma}[density lemma for bounded domain, see Chen--Shaw \cite{CS01}]
Let $\Omega \subset \mathbb{C}^n$ be a bounded pseudoconvex domain with smooth boundary. then $\mathscr{C}_{n,q}^\infty(\overline{\Omega},\mathbb{C}) \cap \Domain(\bar\partial^*)$ is dense in $\Domain(\bar\partial) \cap \Domain(\bar\partial^*)$ in the graph norm $u \mapsto(\|u\|^2+\|\bar{\partial} u\|^2+\|\bar\partial^* u\|^2)^\frac{1}{2}$.
\end{lemma}

By the above lemma, we obtain the following theorem.

\begin{theorem}\label{thm3.3}
      For an arbitrary $(n,q)$ form $u\in \Domain(\bar\partial)\cap\Domain(\bar\partial^*)$, \[
    (t_1+t_3)\|\bar\partial u\|_\phi^2+(t_2+t_3)\|\bar\partial^*_\phi u\|^2_\phi\geq-\frac{1}{4t_3}\|(\bar\partial\eta)^\sharp\lrcorner u\|_\phi^2+\dinn{\i\Theta_{t_3\phi+\eta}\Lambda u,u}_\phi-\frac{1}{4t_2}\|(\partial\eta)^\sharp\lrcorner u\|^2_\phi-\frac{1}{4t_1}\|\bar\partial\eta\wedge u\|^2_\phi.
\]
In particular, if $u\in\Domain(\bar\partial^*)\cap\ker(\bar\partial)$, \[
    (t_2+t_3)\|\bar\partial^*_\phi u\|^2_\phi\geq-\frac{1}{4t_3}\|(\bar\partial\eta)^\sharp\lrcorner u\|_\phi^2+\dinn{\i\Theta_{t_3\phi+\eta}\Lambda u,u}_\phi-\frac{1}{4t_2}\|(\partial\eta)^\sharp\lrcorner u\|^2_\phi.\]
\end{theorem}

\section{Reformulation as a Donnelly--Fefferman type estimate}\label{thelastsection}
Consider Theorem \ref{thm3.3}. Let $t_3=t_2=1$. From the basic equation, assuming $u\in\ker\bar\partial$, we have that, for $u\in\Domain(\bar\partial^*_{0,1;\phi})\cap\ker(\bar\partial)$ \begin{equation}\label{1}
    \begin{split}
    &2\|\bar\partial^*_\phi u\|^2_\phi\geq\dinn{\i\Theta_{\phi+\eta}\Lambda \tilde u,\tilde u}_\phi-\frac{1}{4}\|(\partial\eta)^\sharp\lrcorner \tilde u\|^2_\phi-\frac{1}{4}\|(\bar\partial\eta)^\sharp\lrcorner \tilde u\|^2_\phi\\=&\int_\Omega (\i\frac{\partial^2(\eta+\phi)}{\partial z_i \bar z_j}-\frac{1}{4}\i\frac{\partial\eta}{\partial z_i}\frac{\partial\eta}{\partial \bar z_j})u_i\bar u_j e^{-\phi}\,dV-\frac{1}{4}\int_\Omega |\bar\partial\eta|^2|u|^2\,dV.
\end{split}
\end{equation}

Consider a function $f\in\ker(\bar\partial)$. Then $\kappa: \bar\partial^*_\phi g\mapsto\dinn{f, g}_{\phi}$ is well-defined for $g\in\ker(\bar\partial)$ (see Chen--Shaw \cite{CS01} and Straube \cite{St10}). Let $u\in\Im(\bar\partial^*_\phi)$ and solve $\bar\partial u=f$. We then have that for all $g\in\ker(\bar\partial)\cap\Domain(\bar\partial^*_\phi)$, \[\dinn{u,\bar\partial^*_\phi g}_\phi=\dinn{f,g}_\phi.\] 

Denote $\i\frac{\partial^2(\eta+\phi)}{\partial z_i \bar z_j}-\i\frac{1}{4}\frac{\partial\eta}{\partial z_i}\frac{\partial\eta}{\partial \bar z_j}-\i\frac{1}{4}|\bar\partial\eta|^2$ by $\Xi$ as a pointwise action on $u$. We have that, based on (\ref{1}),  \[\|\bar\partial^*_\phi g\|^2_\phi\geq\frac{1}{2}\int_\Omega \inn{\Xi g, g}[\mathbb{C}] e^{-\phi}\,dV.\] Here, by definition \[\Xi u=(\i\frac{\partial^2(\eta+\phi)}{\partial z_i \bar z_j}-\frac{1}{4}\i\frac{\partial\eta}{\partial z_i}\frac{\partial\eta}{\partial \bar z_j})u_i\bar u_j -\frac{1}{4} |\bar\partial\eta|^2|u|^2\]

Consequently, \[\dinn{u,\bar\partial^*_\phi g}_\phi\leq\dinn{\Xi^{-1}f,f}_\phi^{1/2}\dinn{\Xi g,g}_\phi^{1/2}\leq \sqrt {2}\dinn{\Xi^{-1}f,f}_\phi^{1/2}\|\bar\partial^*_\phi g\|_\phi.\] So, $\|u\|^2_\phi\leq 2\dinn{\Xi^{-1}f,f}[\phi]$, as long as $\i\Theta_{\phi+\eta}\geq \frac{\i}{4}\partial\eta\wedge\bar\partial\eta+\frac{\i}{4}|\bar\partial\eta|^2$. In practice, since $\phi\in\PSH(\Omega)$, as long as $\eta\in\PSH(\Omega)$ with $\i\partial\bar\partial\eta\geq \frac{\i}{4}\partial\eta\wedge\bar\partial\eta+\frac{\i}{4}|\bar\partial\eta|^2$, we can obtain the estimate $\|u\|^2_\phi\leq 2\dinn{\Xi^{-1}f,f}[\phi]$.

In general, we have the following theorem:

\begin{theorem}
Let $\Omega$ be a bounded pseudoconvex domain in $\mathbb{C}^n$ and $E:=\Omega\times\mathbb{C}\rightarrow X$ be a trivial line bundle with fiber metric $e^{-\phi}$ for $\phi\in C^2(\overline{\Omega})$. Assume that there exists $\eta\in C^2(\overline{\Omega})$ and $t_2,t_3\in\mathbb{R}$ so that \[\Xi_{t_2,t_3,\eta}:=\i\frac{\partial^2(\eta+t_3\phi)}{\partial z_i \bar z_j}-\i\frac{1}{4t_2}\frac{\partial\eta}{\partial z_i}\frac{\partial\eta}{\partial \bar z_j}-\i\frac{1}{4t_3}|\bar\partial\eta|^2\geq 0.\] Let $f\in L^2_{(0,1)}(\Omega)\cap\Domain(\bar\partial)$ so that $\bar\partial f=0$, there exists $u\in\Domain(\bar\partial)$ so that \[\|u\|^2_\phi\leq (t_2+t_3)\dinn{\Xi^{-1}_{t_2,t_3,\eta}f,f}[\phi].\]

\end{theorem}

Similarly, we can extend the theorem to $(0, q)$ forms using the definition on Page 371 of Demailly \cite{De12}. Let $\tilde u$ be a $(n,q)$- form which lifts $u$. We also let $\tilde u$ be a $(0,q)$-form if $u$ is a $(n, q)$ form. In other words, we have that $\tilde{\tilde u}=u$ if $u$ is either a $(0,q)$ or a $(n,q)$ from. Define that, in the Euclidean space, \[\begin{split}
    &\Xi_{t_2,t_3,\eta}\tilde u\\=&\i\Theta_{\eta+t_3\phi}\Lambda\tilde u-\i\frac{1}{4t_2}\bar\partial\eta\wedge (\partial\eta)^\sharp \lrcorner\tilde u-\i\frac{1}{4t_3}\partial\eta\wedge(\bar\partial\eta)^\sharp\lrcorner\tilde u\\
    =&\i\Theta_{\eta+t_3\phi}\Lambda\tilde u-\i\frac{1}{4t_2}\bar\partial\eta\wedge (\partial\eta)^\sharp \lrcorner\tilde u-\i\frac{1}{4t_3}|\partial\eta|^2\tilde u,
\end{split}\] where $\Lambda=L^*$ and $L=\i \sum_i d z_i\wedge d\bar z_i$. In the case for a $(0,q)$ form $u$ , we define $\Xi_{t_2,t_3,\eta} u=\widetilde{\Xi_{t_2,t_3,\eta}\tilde u}$. Then Theorem \ref{mainthm} follows similarly to the theorem above.

		\bigskip
	\bigskip
	\noindent {\bf Acknowledgments}. The author greatly appreciates Dr. Bo-Yong Chen's patience with his questions.
	\printbibliography

\end{document}